\documentclass[11pt]{article}

 \usepackage{amsmath,amssymb,amsthm}
    \usepackage[mathscr]{eucal}
\usepackage{pictexwd,dcpic}

\input xy
\xyoption{all}
\xyoption{2cell}
\UseAllTwocells
\input{diagxy}
\newcommand{\avoidwidow}{\par\pagebreak[3]\kern50pt \pagebreak[3]\kern-50pt}

\usepackage{fullpage}

\title{Derivations in Codifferential Categories}

\author{Richard Blute\thanks{Research supported in part by
NSERC.} \footnote{Department of Mathematics and Statistics, University of Ottawa, \textsf{
rblute@uottawa.ca}} \,\,\,\,\,\,\,\,\,\,\,\,\,\,Rory B. B. Lucyshyn-Wright$^*$\footnote{Department of Pure Mathematics and Mathematical Statistics, University of Cambridge, \textsf{rbbl2@dpmms.cam.ac.uk}}\,\,\,\,\,\,\,\,\,\,\,\,\,\,\,\,Keith O'Neill$^*$\footnote{Department of Mathematics and Statistics, University of Ottawa, \textsf{tonei018@uottawa.ca}}}

\newtheorem{thm}{Theorem}[section]

\newtheorem{prop}[thm]{Proposition}
\newtheorem{rem}[thm]{Remark}
\newtheorem{lem}[thm]{Lemma}
\newtheorem{defn}[thm]{Definition}

\newcommand{\rarr}{\rightarrow}
\newcommand{\cC}{\mbox{$\cal C$}}

\newcommand{\sS}{\mbox{$\sf S$}}

\newcommand{\ox}{\otimes}

\newcommand{\del}{\partial}

\newcommand{\SSS}{\textnormal{\textsf{S}}}

\def\pushright#1{{%              set up
    \parfillskip=0pt            % so \par doesnt push \square to left
    \widowpenalty=10000         % so we dont break the page before \square
    \displaywidowpenalty=10000  % ditto
    \finalhyphendemerits=0      % TeXbook exercise 14.32
    \leavevmode                 % \nobreak means lines not pages
    \unskip                     % remove previous space or glue
    \nobreak                    % don't break lines
    \hfil                       % ragged right if we spill over
    \penalty50                  % discouragement to do so
    \hskip.2em                  % ensure some space
    \null                       % anchor following \hfill
    \hfill                      % push \square to right
    {#1}                        % the end-of-proof mark (or whatever)
    \par}}                      % build paragraph

 \def\qed{\pushright{$\Box$}\penalty-700 \smallskip}

\newenvironment{prf}[1]{\begin{trivlist} \item[{\bf ~Proof}#1.]}%
{\qed\end{trivlist}}

\begin{document}

\maketitle

\begin{quote}
 {\bf Dedication.}  The authors dedicate this work to the memory of Jim Lambek.
\end{quote}

\bigskip
\bigskip

\begin{abstract}
{\it Derivations} provide a way of transporting ideas from the calculus of manifolds to algebraic settings where there 
is no sensible notion of limit. In this paper, we consider derivations in certain monoidal categories, called {\it 
codifferential categories}.  Differential categories were introduced as the categorical framework for modelling differential linear logic. The deriving transform of a differential category, which models the differentiation inference rule, is a derivation in the dual category. We here explore that derivation's universality. 

One of the key structures associated to a codifferential category is an {\it algebra modality}. This is a monad $T$ such that each object of the form $TC$ is canonically an associative, commutative algebra. Consequently, every $T$-algebra has a canonical commutative algebra structure, and we show that universal derivations for these algebras can be constructed quite generally. 

It is a standard result that there is a bijection between derivations from an associative algebra $A$ to an $A$-module 
$M$ and algebra homomorphisms over $A$ from $A$ to $A\oplus M$, with $A\oplus M$ being considered as an infinitesimal extension of $A$. We lift this correspondence to our setting by showing that in a codifferential category there is a canonical $T$-algebra structure on $A\oplus M$. We call $T$-algebra morphisms from $TA$ to this $T$-algebra structure {\it Beck $T$-derivations}. This yields a novel, generalized notion of derivation.

The remainder of the paper is devoted to exploring consequences of that definition. Along the way, we prove that the 
symmetric algebra construction  in any suitable symmetric monoidal category provides an example of  codifferential 
structure, and using this, we give an alternative definition for differential and codifferential categories.

\end{abstract}

\pagebreak

\section{Introduction}

The theory of K\"ahler differentials \cite{H,M} provides an analogue of the theory of 
differential forms and all of its various uses in settings other than the usual setting of smooth manifolds. 
They were originally introduced by K\"ahler as an abstract algebraic notion of differential form. One of their 
advantages is that  they can be applied to  varieties which are not also 
smooth manifolds, such as singular varieties in characteristic 0 or arbitrary varieties over a field of characteristic $p$.
In a setting where one does not have access to limits, one can still talk about {\it derivations}. That is to say one 
passes from the variety to its coordinate ring, and then considers a module over that ring. A derivation is then a linear 
map from the algebra to the module satisfying the Leibniz rule. The module of {\it K\"ahler differentials} or {\it K\"ahler 
module} is then a module equipped with a universal derivation. As usual, such a module is unique up to 
isomorphism. 

Since this initial work, the idea of extending differential forms to more and more abstract settings has advanced in a 
number of different directions.  As one important example, we mention the noncommutative differential 
forms that arise in noncommutative geometry \cite{Landi}.

{\em Differential linear logic} \cite{ER1,ER2} arose originally from semantic concerns. Ehrhard \cite{Ehr1,Ehr2}
had constructed several models of linear logic \cite{G} in which the hom-sets had a natural differentiation operator. 
Ehrhard and Regnier then described this operation as a sequent rule and represented it as a construction and a 
rewrite rule for both interaction nets and for $\lambda$-calculus. The corresponding categorical structures were 
introduced in \cite{BCS,BCS2} and called {\it differential categories} and {\it cartesian differential categories}. 
Cartesian differential categories are an axiomatization of the coKleisli category of a differential category. 

The notion of {\it K\"ahler category} \cite{BCPS} began with the observation that the {\it deriving transform}, the key 
feature of differential categories, is a derivation and, under certain assumptions, has a universal property discussed 
below. (Actually, we must work with the dual notion of codifferential category. If we worked with coalgebras and coderivations, we could work in differential categories and all of the following work, suitably {\it op}-ed, would still hold.) It thus seemed likely 
that an abstract monoidal setting in which K\"ahler differential modules could be defined would apply to 
differential categories. In fact, the original paper only partially resolved this issue. In the present paper, we 
provide a much more satisfying answer by generalizing the notion of derivation to take into account all of the 
codifferential structure, thereby establishing a suitable universal property in full generality. 

A K\"ahler category is an additive, symmetric monoidal category with an {\it algebra modality}, i.e. a monad $T$ such 
that each object of the form $TC$ is equipped with a commutative, associative algebra structure and several 
coherence equations hold, and each associative algebra has an object of universal derivations. In essence, we are 
requiring a K\"ahler module for each free $T$-algebra.

The present paper extends the work of \cite{BCPS} in several ways. It is not surprising that, given all the structure at 
hand, one can endow every $T$-algebra with the structure of  a commutative, associative algebra.
We show that in a K\"ahler 
category, one can use the existence of K\"ahler objects for free $T$-algebras to derive K\"ahler objects for all 
algebras\footnote{We realize that the unavoidable use of the word {\it algebra} in two different ways is confusing. The
word algebra without a $T-$ in front of it will always mean commutative, associative algebra.} that arise in this way. 
Thus if the algebra category is monadic over the base, we can derive K\"ahler modules for all algebras by a single 
uniform procedure. These results follow from the M.Sc. thesis of the third author \cite{O}. 

We also tackle the idea of what it means to be a derivation. It is well-known \cite{Bour} that if $A$ is a commutative 
algebra and $M$ is an $A$-module, then there is a canonical algebra structure on $A\oplus M$ such that derivations
from $A$ to $M$ are in bijective correspondence to algebra maps over $A$ from $A$ to $A\oplus M$. Essentially the
algebra $A\oplus M$ is the extension of $A$ by $M$-infinitesimals. This idea was used in a much more general 
setting by Beck \cite{B}.

While this is a straightforward calculation, it has far-reaching generalizations. First we show that in a 
codifferential category, given a $T$-algebra $(A,\nu)$ and a module $M$ over the algebra 
associated to $A$, there is a canonical $T$-algebra structure on $A\oplus M$ which under the passage from $T$-
algebras to algebras yields 
the traditional associative algebra structure on $A\oplus M$ from \cite{B}. We call this $T$-algebra $W(A,M)$.
We then define a {\it Beck $T$-derivation on $A$ valued in $M$} to be a map of $T$-algebras from 
$(A,\nu)$ to $W(A,M)$ in the slice category over $A$.  Beck $T$-derivations can be equivalently given by morphisms $\del\colon A\rarr M$ satisfying a {\it chain rule} condition with respect to $T$.

We show that the symmetric algebra monad yields a
codifferential category in a very general setting and in this case, our notion of Beck $T$-derivation is equivalent to 
the usual notion of derivation. 

We define a {\it module of K\"ahler $T$-differentials} to be an $A$-module with a universal Beck $T$-derivation. 
We then show that the deriving transform in a codifferential category is always universal in this sense.
In fact, every $T$-algebra has a universal $T$-derivation. 
Our analysis also yields an equivalent definition of differential category we believe will be valuable in 
generalizations of this abstract notion of differentiation. For example, it generalizes in a straightforward way to
noncommutative settings. 

We note that in \cite{DK}, Dubuc and Kock define a notion of derivation on an algebra of a {\it Fermat theory}, the latter being a finitary set-based algebraic theory extending the theory of commutative rings and satisfying a certain axiom. It would be interesting to compare their notion with the notion of $T$-derivation defined here in the monoidal context of codifferential categories.

The extension of K\"ahler categories and codifferential categories to noncommutative settings is an important project, and work of this sort has already begun \cite{Co}. In that paper, Cockett has explored 
the implications of demanding for each $T$-algebra $A$ and each $A$-bimodule $M$ a {\it given} $T$-algebra structure on $A\oplus M$ satisfying certain axioms, whereas here we have shown that in the setting of a codifferential category, a $T$-algebra structure on $A\oplus M$ can be defined in terms of the given codifferential structure.

\bigskip

\section{Derivations and categorical frameworks}

This section covers the theory of derivations, both in its classical formulation with respect to algebras over a field 
and several of its more abstract categorical formulations.

\subsection{Classical case}

Derivations were originally considered for associative, commutative algebras over a field and are employed in algebraic geometry and commutative algebra \cite{E,H}.

\begin{defn}{\em Let $k$ be a commutative ring, $A$ a commutative $k$-algebra, and
$M$ an $A$-module. (All modules throughout the paper will be {\it left} modules.)

A \emph{$k$-derivation} from $A$ to $M$ is a $k$-linear map $\del: A\to M$
such that $\del(aa')=a\del(a')+a'\del(a)$.
}\end{defn}
One can readily verify under this definition that $\del(1)=0$ and hence
$\del(r)=0$ for any $r\in k$.

\begin{defn}{\em
Let $A$ be a $k$-algebra. A \emph{module of $A$-differential forms} is
an $A$-module
$\Omega_A$ together with a $k$-derivation
$\del: A\to\Omega_A$  which is universal in the following sense:
For any $A$-module $M$, and for any $k$-derivation $\del': A\to M$,
there exists a unique $A$-module homomorphism $f:\Omega_A\to M$ such
that $\del'=\del; f$.
}\end{defn}

\begin{lem} For any commutative $k$-algebra $A$, a module of $A$-differential forms
exists.
\end{lem}

There are several well-known constructions. The most straightforward, although
the resulting description is not that useful, is obtained by constructing
the free $A$-module generated by the symbols $\{\del a\mid a\in A\}$ divided out
by the evident relations, most significantly $\del(aa')=a\del(a')+a'\del(a)$.

\subsection{Derivations as algebra maps}\label{Alg=Der}

We suppose we are working in the category of vector spaces over a field $k$, that $A$ is a commutative $k$-algebra 
and $M$ an $A$-module. Define an associative, commutative algebra structure on $A\oplus M$ by 

\[(a,m)\cdot(a',m')=(aa',am'+a'm)\]

It is evident that this is associative, commutative and unital. We will refer to this algebra structure as the {\it 
infinitesimal extension of $A$ by 
$M$}. But its interest comes from the following observation. 

\begin{lem}\label{Beck-used-it}
There is a bijective correspondence between $k$-derivations from $A$ to $M$ and $k$-algebra homomorphisms from $A$ to 
$A\oplus M$ which are the identity in the first component. Or more succinctly:

\[Der_k(A,M)\cong Alg/A(A,A\oplus M)\]

Here, $Alg/A$ is the slice category of objects over $A$ in the category {\it Alg} of $k$-algebras. 
\end{lem}

We also note that it is straightforward to lift this result to the level of additive symmetric monoidal categories, see Section \ref{sec:cat_struct}. The notions of commutative algebra and module are expressible in any symmetric monoidal category. Once one has
 additive structure then the notion of derivation is definable as well. The correspondence of Lemma \ref{Beck-used-it}
then extends to this more general setting. 
Lemma \ref{Beck-used-it} also provided Jon Beck \cite{B} a starting point for a far-reaching generalization
of the notion of derivation for the purposes of cohomology theory.
One of the primary contributions of this paper is to lift the correspondence of Lemma \ref{Beck-used-it} to the level of codifferential categories. The fact that these 
ideas continue to hold at this level is testament to the importance of Beck's ideas about cohomology.

\subsection{Categorical structure}\label{sec:cat_struct}

It is a standard observation \cite{CWM,K} that the notions of algebra (monoid) and module over an algebra 
make sense in any monoidal category and the notion of commutative algebra makes sense in any symmetric monoidal category.  But to discuss derivations for an algebra we also need additive structure.

\begin{defn}\label{def:ordinary_dern}

{\em 
\begin{enumerate}
\item
A symmetric monoidal category $\mathcal{C}$\
is \emph{additive} if it is
enriched over commutative monoids and the tensor functor is additive in both variables.\footnote{In particular, we only need addition and unit on
Hom-sets, rather than abelian group structure.}.  
\item Let $(A,m_A,e_A)$ be an algebra in an additive symmetric monoidal category\footnote{We will use the notation $m_A$ and $e_A$ for the multiplication and unit for $A$.}, and
$M=\langle M,\bullet_M:A\ox M\to M\rangle$ an $A$-module. Then a \emph{derivation to $M$}
is an arrow $\partial: A\to M$ such that (with $m$ being the multiplication)
\[m;\partial=c;1\otimes \partial;\bullet_M+1\otimes \partial;\bullet_M
\mbox{\em~~~~~~ and ~~~~~~}
\partial(1)=0\]

\end{enumerate}}
\end{defn}

\begin{rem}\label{rem:derns_alg_maps_mon}\textnormal{
We note that Lemma \ref{Beck-used-it} holds at this level of
generality as well.  Indeed, given a commutative algebra $A$ in an
additive symmetric monoidal category $\cC$ with finite coproducts
(equivalently, finite biproducts) and an $A$-module $M$, we can equip
$A \oplus M$ with the structure of a commutative algebra \cite{BCPS}.  
Derivations $A \rightarrow M$ then correspond to maps $A
\rightarrow A \oplus M$ in the slice category $Alg/A$ over $A$ in the
category $Alg$ of commutative algebras in $\cC$ 
\cite{BCPS}.  As noted in \cite[\S 4.2]{BCPS}, every map of $A$-modules
$h:M \rightarrow N$ determines an algebra map $1 \oplus h:A \oplus M
\rightarrow A \oplus N$, whence each derivation $\partial:A
\rightarrow M$ determines a composite derivation $A
\xrightarrow{\partial} M \xrightarrow{h} N$.  Further, given a map of
commutative algebras $g:A \rightarrow B$, each $B$-module $N$
determines an $A$-module $N_A$, the \textit{restriction of scalars of
$N$ along $g$}, consisting of the object $N$ of $\cC$ equipped with
the composite $A$-action
$$A \otimes N \xrightarrow{g \otimes 1} B \otimes N
\xrightarrow{\bullet_N} B\;.$$
Moreover, given an algebra map $g:A \rightarrow B$ and a derivation
$\partial:B \rightarrow N$, the composite $A \xrightarrow{g} B
\xrightarrow{\partial} N$ is a derivation $A \rightarrow N_A$.}
\end{rem}

 As for most algebraic structures, when one adds in an appropriate notion of universality, the result is a very powerful
mathematical object. For derivations, we obtain the module of {\it K\"ahler differentials} or {\it K\"ahler module}.
We cite \cite{H,M} for calculations and examples. 

\begin{defn}{\em 
Let \cC\ be an additive symmetric monoidal
category and let $A$ be a commutative algebra in \cC. A {\em module of K\"ahler differentials} is an $A$-module
$\Omega_A$ together with a derivation $\del: A\to \Omega_A$, such that for every $A$-module M, and for every 
derivation $\del': A\to M$,
there exists a unique $A$-module map
$h:\Omega_A\to M$ such that  $\del;h=\del'$.
\[
\xymatrix{A \ar[r]^{\partial} \ar[rd]_{\partial'} & \Omega_A
\ar@{..>}[d]^h \\ & M}
\]

}\end{defn}

An axiomatization of a very different sort which attempted to capture the process of differentiation axiomatically
is the theory of {\it differential categories} \cite{BCS}. Since in this paper we wish to work with algebras and 
derivations as opposed to coalgebras and coderivations, we work in the dual theory of {\it codifferential categories}.

 \begin{defn}{\em
An  \emph{algebra modality} on a symmetric monoidal category \cC\ consists of a monad $(T,\mu,\eta)$ on \cC, 
and for each object $C$ in $\mathcal{C}$, a pair of morphisms (note we are denoting the tensor unit by $k$)

$$m : T(C)\otimes T(C)\to T(C), \hspace*{1cm} e: k\to T(C)$$

\noindent making $T(C)$ a commutative algebra
such that this family of associative algebra structures satisfies evident naturality conditions \cite{BCPS}.
}\end{defn}

\begin{defn}\label{def:codiff_cat}
{\em
An additive symmetric monoidal category with an algebra modality is a
\emph{codifferential category} if it is also equipped with a
\emph{deriving transform}\footnote{We use the terminology of a
{\em deriving transform} in both differential and
codifferential categories.}, {\em i.e.} a transformation natural in $C$
$$d_{T(C)}: T(C)\to T(C)\otimes C$$
satisfying the following four equations\footnote{For simplicity, we
write as if  the monoidal structure is strict.}:
\begin{description}
\item [(d1)] $e;d=0$ \,\,\, {\sl (Derivative of a constant is 0.)}
\item [(d2)]  $m;d=(1\otimes d);(m\otimes 1) +
(d\otimes 1);c;(m\otimes 1)$ (where $c$ is
the appropriate symmetry)  \,\,\, {\sl (Leibniz Rule)}
\item [(d3)] $\eta;d=e\otimes 1$  \,\,\, {\sl (Derivative of a linear function is constant.)}
\item [(d4)]$\mu;d=d;\mu\otimes d;m\otimes 1$ \,\,\, {\sl (Chain Rule)}
\end{description}

}\end{defn}

We make the following evident observation, noting that the morphism $u^{TC}_C:=e\ox 1\colon C=k\ox C\rarr T(C)\ox C$
exhibits $T(C)\ox C$ as the free $T(C)$-module on $C$. 

\begin{lem}
When $T(C)\ox C$ is considered as the free $T(C)$-module generated by $C$, then the above deriving transform 
is a derivation.
\end{lem}

This leaves the question of its universality. We know there is a universal property for the object 
$T(C)\ox C$ as the free $T(C)$-module generated by $C$. Is this sufficient to guarantee the universality necessary to be a K\"ahler module?
With this question in mind, the paper \cite{BCPS} introduced the notion of a {\it K\"ahler category} but only partially
answered this question. 

\begin{defn}{\em 
A {\em K\"ahler category} is an additive symmetric monoidal
category with
\begin{itemize}
\item a monad $T$,
\item a (commutative) algebra modality for $T$,
\item for all objects $C$, a $T(C)$-module of K\"ahler differential forms,
satisfying the universal property of a K\"ahler module.
\end{itemize}}
\end{defn}

Thus the previous question can be formulated as whether every codifferential category is a K\"ahler category. The 
original paper \cite{BCPS} had a partial answer to this question. 
In the present paper, we give a much more satisfying answer to this question. The key is to abstract even further the 
notion of derivation. We use ideas from Jon Beck's remarkable thesis \cite{B}. This will be covered in Section \ref{beckt}.

\subsection{Universal derivations for $T$-algebras}

In a category with an algebra modality we may endow each $T$-algebra with the structure of a commutative algebra, 
in such a way that the structure map of the $T$-algebra is a morphism of associative algebras. Since universal derivations are a priori 
only defined for the algebras arising axiomatically in a K\"ahler category, it is natural to ask if universal derivations 
from these 
new associative algebras exist and, if so, how they are constructed. We examine this issue now and demonstrate 
that there is a very pleasing answer. The construction of such K\"ahler modules is from the third author's M.Sc. 
thesis \cite{O}. We first note the following procedure for assigning algebra structure to $T$-algebras.

\begin{thm} \label{T-Alg-to-Alg}

Let \cC\ be a symmetric monoidal category equipped with an algebra modality $T$. The following 
construction determines a functor from the category of $T$-algebras to the category of commutative associative 
algebras in \cC. Let $(A, \nu)$ be a $T$-algebra in such a category. Define the multiplication for an algebra structure on $A$ by the formula

\[A\ox A\stackrel{\eta\ox\eta}{-\!\!\!\!\!-\!\!\!\!\!-\!\!\!\longrightarrow}TA\ox TA
\stackrel{m}{-\!\!\!\!\!-\!\!\!\!\!-\!\!\!\longrightarrow} TA\stackrel{\nu}{-\!\!\!\!\!-\!\!\!\!\!-\!\!\!\longrightarrow}A\]

\noindent with unit given by 

\[k\stackrel{e}{-\!\!\!\!\!-\!\!\!\!\!-\!\!\!\longrightarrow}TA\stackrel{\nu}{-\!\!\!\!\!-\!\!\!\!\!-\!\!\!\longrightarrow}A\]

\noindent In particular, every map of $T$-algebras becomes an associative algebra map. 

Also note that if we apply this construction to the free $T$-algebra $(TA,\mu)$, we get back the original 
associative algebra $(TA,m,e)$.

\end{thm}

\begin{defn}\em Let \cC\ be an additive symmetric monoidal category. Let $A$ and $B$ be algebras with universal 
derivations as in the diagram below. Let $f\colon A\rarr B$ be an algebra homomorphism.
Define \em $\Omega_{f}:\Omega_{A} \to \Omega_{B}$ \em to be the unique morphism of $A$-modules making
%FUNCTOR

	$$\bfig
	\morphism(0,0)|b|/{>}/<750,0>[A`B;f]
	\morphism(0,0)|l|/{>}/<0,750>[A`\Omega_{A};d_{A}]
	\morphism(750,0)|r|/{>}/<0,750>[B`\Omega_{B};d_{B}]
	\morphism(0,750)|a|/{>}/<750,0>[\Omega_{A}`\Omega_{B};\Omega_{f}]
	\efig$$
%FUNCTOR
commute, which exists by universality of $d_{A}$. One can verify that $\Omega_{(-)}$ is functorial.
\end{defn}

\smallskip

The existence of K\"ahler modules  for free $T$-algebras entails that K\"ahler modules  for arbitrary $T$-algebras 
can be obtained by taking a quotient, as is seen in the following theorem. 

\begin{thm} Defining $\Omega_{A,\nu}$ as the following coequalizer

$$\bfig
\morphism(0,0)|a|/@{>}@<+2pt>/<500,0>[\Omega_{T^2A}`\Omega_{TA};\Omega_\mu]
\morphism(0,0)|b|/@{>}@<-2pt>/<500,0>[\Omega_{T^2A}`\Omega_{TA};\Omega_{T\nu}]
\morphism(500,0)|a|/{>}/<500,0>[\Omega_{TA}`\Omega_{A,\nu};\Omega_\nu]
\morphism(1000,0)|a|/{}/<0,0>[\Omega_{A,\nu}`;]
\efig$$

gives us the module of K\"ahler differentials for $T$-algebra $(A,\nu)$.
\end{thm}

This result was in the M.Sc. thesis of the third author \cite{O}. We do not give a proof of this result here as it can be 
obtained in a method similar to Theorem \ref{main}. We also note that, under suitable hypotheses, the existence of K\"ahler modules for arbitrary commutative algebras follows from Theorem \ref{main}.

\section{The symmetric algebra monad}

The most canonical example of an algebra modality is the symmetric algebra construction.  This construction 
as applied to the category of vector spaces gives one of the most basic examples of a codifferential category. In this 
case,  elements of the symmetric algebra are essentially polynomials, which are differentiated in the evident way. 
A similar construction works on the category of sets and relations. What we observe here is that the symmetric 
algebra construction provides examples of codifferential categories in a much more general setting. 

First, we need to explore a theme which will be the centrepiece of the last sections of the paper. This is the 
idea of viewing derivations as algebra homomorphisms.

\begin{rem}\label{monadic}
For the remainder of this section, we assume $\mathcal{C}$ is an additive symmetric monoidal category with finite 
coproducts and reflexive coequalizers, 
the latter of which are preserved by the tensor product in each variable. Let $Alg$ be the category of commutative 
algebras in $\mathcal{C}$, and suppose that the forgetful functor $Alg \to\mathcal{C}$ has a left adjoint. The 
resulting 
adjunction is then monadic; denote its induced monad by $\mathcal{S} = (S,\eta^S,\mu^S)$, so 
that $Alg \cong \mathcal{C}^\mathcal{S}$, and we henceforth identify these categories. See \cite{CWM} for details.
\end{rem}

\subsection{Structure related to the symmetric algebra}

We will also need the following straightforward observation:
\begin{prop}\label{thm:alg_modality_via_mnd_mor} The (commutative) algebra modalities on \cC\ are in bijective 
correspondence to pairs $(T,\psi)$, where $T$ is a 
monad and $\psi$ is a monad morphism $\psi\colon \sS\rarr T$. Such a morphism induces a functor 

\[F_\psi\colon T\mbox{\cal -Alg}\rarr\sS\mbox{\cal -Alg}\]

Furthermore, the map $\psi_C\colon \sS C\rarr TC$ is a map of algebras.

\end{prop}

\subsection{Codifferential structure}

\begin{defn}{\em 
Given an object $C$ in \cC, recall that $SC\ox C$ is the free $SC$-module on $C$. Hence by Remark
\ref{rem:derns_alg_maps_mon}, the direct sum $SC \oplus (SC \otimes C)$ carries the structure of an algebra, and derivations 
 $SC \rarr (SC \otimes C)$ correspond to algebra homomorphisms $SC \rarr SC \oplus (SC \otimes C)$
 whose first coordinate is the identity. But since $SC$ is the free algebra on $C$, the latter correspond to
 morphisms $C \rarr SC \oplus (SC \otimes C)$ whose first coordinate is $\eta\colon C\rarr SC$.

So let $d_{SC}:SC \to SC \otimes C$ be the derivation corresponding to the algebra homomorphism $SC \to SC 
\oplus (SC \otimes C)$ given on generators as $\begin{pmatrix} \eta_C \\ u_C \end{pmatrix}:C \to SC \oplus 
(SC \otimes C)$, where $u_C$ is the map $u_C\colon C\cong k\ox C\stackrel{e\ox 1}{-\!\!\!\!\!-\!\!\!\!\!-\!\!\!\longrightarrow}SC\ox C$.
}\end{defn}

\begin{thm}\label{thm:smalg_mnd_codiff}
$(\mathcal{C},S,d)$ is a codifferential category.
\end{thm}
\begin{prf}{}
$S$ is a commutative algebra modality on $\mathcal{C}$. Since each $d_{SC}$ is by definition a derivation, the Leibniz rule holds and precomposing $d_{SC}$ by $e_{SC}$ is the zero map. By the definition of $d_{SC}$, 
	\begin{displaymath}
	\eta_C;\begin{pmatrix} 1_{SC} \\ d_{SC} \end{pmatrix} = \begin{pmatrix} \eta_C \\ u_C \end{pmatrix}:C \to SC \oplus (SC \otimes C)
	\end{displaymath}
 so that 
	\begin{displaymath}
	\eta_C;d_{SC} = \eta_C;\begin{pmatrix} 1_{SC} \\ d_{SC} \end{pmatrix};\pi_2 =\begin{pmatrix} \eta_C \\ u_C \end{pmatrix};\pi_2 = u_C
	\end{displaymath}
and consequently (d3) holds.

It remains only to demonstrate naturality of $d$ and adherence to the chain rule condition. For naturality, consider a map $f:C \to D$ in $\mathcal{C}$; naturality of $d$ is equivalent to the commutativity of the following square:

$$\bfig
\morphism(0,0)|a|/{>}/<1000,0>[SC`SC \oplus (SC \otimes C);\begin{pmatrix} 1 \\ d_{SC} \end{pmatrix}]
\morphism(0,0)|l|/{>}/<0,-400>[SC`SD;Sf]
\morphism(1000,0)|r|/{>}/<0,-400>[SC \oplus (SC \otimes C)`SD \oplus (SD \otimes D);Sf \oplus (Sf \otimes f)]
\morphism(0,-400)|b|/{>}/<1000,0>[SD`SD \oplus (SD \otimes D);\begin{pmatrix} 1 \\ d_{SD} \end{pmatrix}]
\efig$$

\noindent Since each morphism in the square is an algebra morphism, commutativity of this square may be demonstrated by showing that the square is commutative when preceded by $\eta_C:C \to SC$. By naturality of $\eta$ and definition of $d_D$ we have on the left: 
	\begin{displaymath}
	\eta_C;Sf;\begin{pmatrix} 1 \\ d_{SD} \end{pmatrix} = f;\eta_D;\begin{pmatrix} 1 \\ d_{SD} \end{pmatrix} = f;\begin{pmatrix} \eta_D \\ u_{D} 			\end{pmatrix}
	\end{displaymath}
 By naturality of $\eta$ and $u$ and by definition of $d$ we have on the right: 
	\begin{displaymath}
	\eta_C; \begin{pmatrix} 1 \\ d_{SC} \end{pmatrix};Sf \oplus (Sf \otimes f) = \begin{pmatrix} \eta_C \\ u_{SC} \end{pmatrix};Sf \oplus (Sf 			\otimes f) = \begin{pmatrix} \eta_C;Sf \\ u_{SC};Sf \otimes f \end{pmatrix} = f;\begin{pmatrix} \eta_D \\ u_{D} \end{pmatrix}
	\end{displaymath}
and so naturality of $d$ is established.

To show that $d$ adheres to the chain rule, it is necessary and sufficient to show that the following square commutes

$$\bfig
\morphism(0,0)|a|/{>}/<2000,0>[S^2C`SC;\mu_C]
\morphism(0,0)|l|/{>}/<0,-400>[S^2C`S^2C \otimes SC;d_{S^2C}]
\morphism(2000,0)|r|/{>}/<0,-400>[SC`SC \otimes C;d_{SC}]
\morphism(0,-400)|b|/{>}/<1000,0>[S^2C \otimes SC`SC \otimes SC \otimes C;\mu_C \otimes d_{SC}]
\morphism(1000,-400)|b|/{>}/<1000,0>[SC \otimes SC \otimes C`SC \otimes C;m_{SC} \otimes 1]
\efig$$

When preceded by $\eta_{SC}$, commutativity of the resultant diagram is established by a routine verification. In order to show that this verification suffices, it must be shown that both paths in the above diagram yield derivations
when preceded by $\eta_{SC}$; the correspondence between derivations and morphisms of algebras then enables the utilization of the universal property of $\eta$ to deduce that the associated morphisms of algebras are equal.

Since $\mu_C$ is an associative algebra homomorphism, $\mu_C;d_{SC}$ is a derivation with respect to the $S^2C$-module structure that $SC\ox C$ acquires by restriction of scalars along $\mu_C$. As for the counterclockwise composite, the following computation demonstrates that it adheres to the Leibniz rule
\begin{align}
&m_{S^2C};d_{S^2C};\mu_C \otimes d_{SC};m_{SC} \otimes 1 \notag \\
&= (1\otimes d_{S^2C} + c;1\otimes d_{S^2C});m_{S^2C} \otimes 1;\mu_C \otimes d_{SC};m_{SC}\otimes 1 \notag \\
&= (1\otimes d_{S^2C} + c;1\otimes d_{S^2C});\mu_C \otimes \mu_C \otimes 1;m_{SC} \otimes 1;1 \otimes d_{SC};m_{SC} \otimes 1 \notag \\
&=(1\otimes (d_{S^2C};\mu_C \otimes d_{SC}) + c;1\otimes (d_{S^2C};\mu_C \otimes d_{SC}));\mu_C \otimes 1 \otimes 1 \otimes 1;m_{SC} \otimes 1 \otimes  1;m_{SC} \otimes 1 \notag \\
&= (1\otimes (d_{S^2C};\mu_C \otimes d_{SC};m_{SC} \otimes 1) + c;1\otimes (d_{S^2C};\mu_C \otimes d_{SC};m_{SC} \otimes 1));\mu_C \otimes 1 \otimes 1;m_{SC} \otimes 1 \notag
\end{align}

\noindent That the counterclockwise composite is $0$ when preceded by $e_{S^2C}$ is immediate, and the proof is complete.
\end{prf}

\section{Beck $T$-derivations}\label{beckt}

We now explore what we consider to be the main contribution of this paper.
The first step in this project is the following theorem, due to the second author. It lifts the 
correspondence between derivations and algebra homomorphisms to the level of $T$-algebras.
Throughout this section, we assume that \cC\ has finite coproducts.  

\begin{thm}
Let \cC\ be a codifferential category with finite coproducts. Let $(A,\nu)$ be a $T$-algebra and $M$ a module over its associated algebra.
Then $(A\oplus M,\beta)$ is a $T$-algebra with $\beta\colon T(A\oplus M)\rarr A\oplus M$ defined as follows. 

Evidently we need maps to $A$ and to $M$ which we define as follows:

\[\beta_1\colon T(A\oplus M) \stackrel{T\pi_1}{-\!\!\!\!\!-\!\!\!\!\!-\!\!\!\longrightarrow}
TA\stackrel{\nu}{-\!\!\!\!\!-\!\!\!\!\!-\!\!\!\longrightarrow} A\]

\[\beta_2\colon  T(A\oplus M) \stackrel{d}{-\!\!\!\!\!-\!\!\!\!\!-\!\!\!\longrightarrow}
T(A\oplus M)\ox (A\oplus M)\stackrel{T(\pi_1)\ox\pi_2}{-\!\!\!\!\!-\!\!\!\!\!-\!\!\!\longrightarrow}T(A)\ox M
\stackrel{\nu\ox 1}{-\!\!\!\!\!-\!\!\!\!\!-\!\!\!\longrightarrow}A\ox M\stackrel{\bullet}{-\!\!\!\!\!-\!\!\!\!\!-\!\!\!\longrightarrow}M\]

\end{thm}

\begin{prf}{} The following four diagrams capture all of the necessary equations.

$$\bfig
\morphism(0,0)|b|/{>}/<600,0>[T^2A`TA;Ta]
\morphism(0,0)|l|/{>}/<0,-400>[T^2A`TA;\mu]
\morphism(600,0)|r|/{>}/<0,-400>[TA`A;a]
\morphism(0,-400)|b|/{>}/<600,0>[TA`A;a]
\morphism(-600,-400)|b|/{>}/<600,0>[T(A\oplus M)`TA;T\pi_1]
\morphism(-600,500)|l|/{>}/<0,-900>[T^2(A\oplus M)`T(A\oplus M);\mu]
\morphism(-600,500)|l|/{>}/<600,-500>[T^2(A\oplus M)`T^2A;T^2\pi_1]
\morphism(-600,500)|r|/{>}/<1200,-500>[T^2(A\oplus M)`TA;T\beta_1]
\morphism(-600,500)|a|/{>}/<1200,0>[T^2(A\oplus M)`T(A\oplus M);T\beta]
\morphism(600,500)|r|/{>}/<0,-500>[T(A\oplus M)`TA;T\pi_1]
\morphism(600,500)|r|/{@{>}@/^3em/}/<0,-900>[T(A\oplus M)`A;\beta_1]
\morphism(-600,-400)|b|/{@{>}@/_2em/}/<1200,0>[T(A\oplus M)`A;\beta_1]

\morphism(1400,500)|l|/{>}/<0,-500>[A\oplus M`A;\pi_1]
\morphism(1400,500)|a|/{>}/<600,0>[A\oplus M`T(A\oplus M);\eta]
\morphism(1400,0)|a|/{>}/<600,0>[A`TA;\eta]
\morphism(2000,500)|r|/{>}/<0,-500>[T(A\oplus M)`TA;T\pi_1]
\morphism(1400,0)|r|/{>}/<300,-500>[A`A;1]
\morphism(2000,0)|r|/{>}/<-300,-500>[TA`A;a]
\morphism(1400,500)|l|/{@{>}@/_4em/}/<300,-1000>[A\oplus M`A;\pi_1]
\morphism(2000,500)|r|/{@{>}@/^4em/}/<-300,-1000>[T(A\oplus M)`A;\beta_1]

\morphism(-350,-100)|a|/{}/<0,0>[(nat \phantom{i} \mu)`;]
\morphism(260,-200)|a|/{}/<0,0>[(T\phantom{i}alg)`;]

\morphism(1700,300)|a|/{}/<0,0>[(nat\phantom{i} \eta)`;]
\efig$$

$$\bfig
\morphism(-100,0)|a|/{>}/<1300,0>[(T(A\oplus M))^{\otimes 2} \otimes(A\oplus M)`(TA)^{\otimes 2}\otimes M;T\pi_1 \otimes T\pi_1 \otimes \pi_2]
\morphism(1200,0)|a|/{>}/<750,0>[(TA)^{\otimes 2}\otimes M`A^{\otimes 2}\otimes M;a \otimes a \otimes 1]
\morphism(1950,0)|a|/{>}/<750,0>[A^{\otimes 2}\otimes M`A \otimes M;1 \otimes \bullet]
\morphism(-100,0)|l|/{>}/<0,-500>[(T(A\oplus M))^{\otimes 2} \otimes(A\oplus M)`T(A \oplus M)\otimes (A\oplus M);m_{T(A\oplus M)} \otimes 1]
\morphism(1200,0)|l|/{>}/<0,-500>[(TA)^{\otimes 2}\otimes M`TA \otimes M;m_{TA} \otimes 1]
\morphism(1950,0)|r|/{>}/<0,-500>[A^{\otimes 2} \otimes M`A \otimes M;m_A \otimes 1]
\morphism(2700,0)|r|/{>}/<0,-500>[A\otimes M`M;\bullet]
\morphism(-100,-500)|b|/{>}/<1300,0>[T(A \oplus M) \otimes (A \oplus M)`TA \otimes M;T\pi_1 \otimes \pi_2]
\morphism(1200,-500)|b|/{>}/<750,0>[TA \otimes M`A \otimes M;a\otimes 1]
\morphism(1950,-500)|b|/{>}/<750,0>[A \otimes M`M;\bullet]
\morphism(-950,-500)|a|/{>}/<850,0>[T(A\oplus M)`T(A\oplus M)\otimes (A\oplus M);d]
\morphism(-100,1300)|l|/{>}/<0,-1300>[T^2(A \oplus M) \otimes T(A \oplus M)`(T(A\oplus M))^{\otimes 2} \otimes (A \oplus M); \mu \otimes d]
\morphism(1200,900)|l|/{>}/<0,-900>[T^2A \otimes T(A\oplus M) \otimes (A \oplus M)`(TA)^{\otimes 2}\otimes M;\mu \otimes T\pi_1 \otimes \pi_2]
\morphism(-100,1300)|l|/{>}/<1300,-400>[T^2(A \oplus M) \otimes T(A \oplus M)`T^2A \otimes T(A\oplus M) \otimes (A \oplus M);T^2\pi_1 \otimes d]
\morphism(1950,500)|r|/{>}/<0,-500>[(TA)^{\otimes 2} \otimes M`A^{\otimes 2} \otimes M;a \otimes a \otimes 1]
\morphism(1200,900)|r|/{>}/<750,-400>[T^2A \otimes T(A\oplus M) \otimes (A \oplus M)`(TA)^{\otimes 2} \otimes M;Ta \otimes T\pi_1 \otimes \pi_2]
\morphism(-100,1300)|l|/{>}/<2800,-300>[T^2(A \oplus M) \otimes T(A \oplus M)`TA \otimes M;T\beta_1 \otimes \beta_2]
\morphism(-100,1300)|a|/{>}/<2500,0>[T^2(A \oplus M) \otimes T(A \oplus M)`T(A\oplus M) \otimes (A\oplus M);T\beta \otimes \beta]
\morphism(2400,1300)/{}/<0,0>[T(A\oplus M) \otimes (A\oplus M)`;]
\morphism(2700,1300)|l|/{>}/<0,-300>[`TA \otimes M;T\pi_1 \otimes \pi_2]
\morphism(2700,1000)|l|/{>}/<0,-1000>[TA \otimes M`A\otimes M;a \otimes 1]
\morphism(2700,1800)|l|/{>}/<0,-400>[T(A\oplus M)`;d]
\morphism(2700,1800)|l|/{@{>}@/^3em/}/<0,-2300>[T(A\oplus M)`M;\beta_2]
\morphism(-950,1800)|a|/{>}/<3650,0>[T^2(A\oplus M)`T(A\oplus M);T\beta]
\morphism(-950,1800)|l|/{>}/<0,-2300>[T^2(A \oplus M)`T(A \oplus M);\mu]
\morphism(-950,1800)|r|/{>}/<850,-500>[T^2(A \oplus M)`T^2(A \oplus M) \otimes T(A \oplus M);d]
\morphism(-950,-500)|b|/{@{>}@/_3em/}/<3650,0>[T(A\oplus M)`M;\beta]
\morphism(900,1600)/{}/<0,0>[(nat\phantom{i} d)`;]
\morphism(-500,800)/{}/<0,0>[(chain\phantom{i} rule)`;]
\morphism(400,-250)/{}/<0,0>[(alg\phantom{i}hom)`;]
\morphism(1600,-250)/{}/<0,0>[(alg\phantom{i}hom)`;]
\morphism(2460,-250)/{}/<0,0>[(act)`;]
\morphism(400,400)/{}/<0,0>[(nat\phantom{i} \mu)`;]
\morphism(1500,400)/{}/<0,0>[(T\phantom{i}alg)`;]
\efig$$

$$\bfig
\morphism(0,0)|a|/{>}/<1100,0>[T(A\oplus M) \otimes (A\oplus M)`TA \otimes M;T\pi_1 \otimes \pi_2]
\morphism(0,-500)|l|/{>}/<0,500>[k \otimes (A\oplus M)`T(A \oplus M) \otimes (A\oplus M);e\otimes 1]
\morphism(0,-500)|b|/{>}/<1100,0>[k\otimes (A\oplus M)`k\otimes (A\oplus M);1]
\morphism(1100,-500)|l|/{>}/<0,500>[k\otimes (A\oplus M)`TA \otimes M;e\otimes \pi_2]

\morphism(-900,0)|a|/{>}/<900,0>[T(A\oplus M)`T(A\oplus M) \otimes (A\oplus M);d]
\morphism(-900,-500)|l|/{>}/<0,500>[A\oplus M`T(A\oplus M);\eta]
\morphism(-900,-500)|a|/{>}/<900,0>[A\oplus M`k\otimes (A\oplus M);\cong]

\morphism(1100,0)|a|/{>}/<750,0>[TA \otimes M`A\otimes M;a\otimes 1]
\morphism(1100,-500)|b|/{>}/<750,0>[k\otimes (A\oplus M)`k\otimes (A\oplus M);1]
\morphism(1850,-500)|r|/{>}/<0,500>[k\otimes (A\oplus M)`A\otimes M;e\otimes \pi_2]

\morphism(1850,0)|a|/{>}/<600,0>[A\otimes M`M;\bullet]
\morphism(1850,-500)|a|/{>}/<600,0>[k\otimes (A\oplus M)`A\oplus M;\cong]
\morphism(2450,-500)|r|/{>}/<0,500>[A\oplus M`M;\pi_2]

\morphism(-900,0)|a|/{@{>}@/^4em/}/<3350,0>[T(A\oplus M)`M;\beta_2]
\morphism(-900,-500)|b|/{@{>}@/_4em/}/<3350,0>[A\oplus M`A\oplus M;1]

\morphism(425,-300)|a|/{}/<0,0>[(T\pi_1\phantom{i}alg\phantom{i}hom)`;]
\morphism(1500,-300)|a|/{}/<0,0>[(a\phantom{i}alg\phantom{i}hom)`;]
\efig$$
\end{prf}

\begin{defn}{\em
We denote this $T$-algebra by $W(A,M)=\langle A\oplus M,\beta^{AM}\rangle$.
}\end{defn}

The following result is straightforward. 

\begin{lem}
Let $(A,\nu)$ be a $T$-algebra. Let $M$ an $A$-module. Then $\pi_1\colon A\oplus M\rarr A$ is a map of $T$-algebras, where $A\oplus M$ is given the $T$-algebra structure just defined.
\end{lem}

We also note that the algebra associated to this $T$-algebra under the process of Theorem \ref{T-Alg-to-Alg}
coincides with the algebra structure associated to $A\oplus M$ in Remark
\ref{rem:derns_alg_maps_mon}.

\begin{prop}
Let $(A,\alpha)$ be a $T$-algebra in $\mathcal{C}$ and let $M$ be an $A$-module. Then the commutative algebra structure carried by the $T$-algebra $A \oplus M$ coincides with the commutative algebra structure on $A\oplus M$ described in Remark
\ref{rem:derns_alg_maps_mon}.   
\end{prop}

\begin{prf}{}
Since $\beta^{AM}$ is an algebra homomorphism the multiplication associated to $W(A,M)$ is
	\begin{displaymath}
	m_{W(A,M)} = \eta_{A \oplus M} \otimes \eta_{A \oplus M};m_{T(A \oplus M)};\beta^{AM}
	\end{displaymath}
Since $\pi_1:W(A,M) \to A$ is a $T$-homomorphism and hence an algebra homomorphism, $m_{W(A,M)};\pi_1 = \pi_1 \otimes \pi_1;m_A$ and so the first component of $m_{W(A,M)}$ is given as in Remark
\ref{rem:derns_alg_maps_mon}.

The second component is the composite 
	\begin{displaymath}
	\eta \otimes \eta;m_{T(A \oplus M)};d_{T(A \oplus M)};T\pi_1 \otimes \pi_2;\alpha \otimes 1;\bullet
	\end{displaymath}

\noindent Calculate as follows:

\begin{align}
	&\eta_{A\oplus M} \otimes \eta_{A\oplus M};m_{T(A\oplus M)};d_{T(A\oplus M)};T\pi_1 \otimes \pi_2;\alpha \otimes 1;\bullet \notag \\
	&= \eta_{A\oplus M} \otimes \eta_{A\oplus M};(1 \otimes d_{T(A\oplus M)} + c;1 \otimes d_{T(A \oplus M)});m_{T(A\oplus M)}\otimes 1;S\pi_1 \otimes \pi_2;\alpha \otimes 1;\bullet \notag \\
	&= (\eta_{A\oplus M} \otimes (\eta_{A\oplus M};d_{T(A\oplus M)}) + c;\eta_{A\oplus M} \otimes ( \eta_{A\oplus M};d_{A \oplus M}));(T\pi_1;\alpha) \otimes (T\pi_1;\alpha) \otimes \pi_2;m_A \otimes 1;\bullet \notag \\
	&= (1 \otimes (\eta_{A\oplus M};d_{A\oplus M}) + c;1 \otimes (\eta_{A\oplus M};d_{A \oplus M}));(\pi_1;\eta_A;\alpha) \otimes (T\pi_1;\alpha) \otimes \pi_2; m_A \otimes 1;\bullet \notag \\
	&= (1 + c);1 \otimes e_{A \oplus M} \otimes 1; \pi_1 \otimes (T\pi_1;\alpha) \otimes \pi_2;m_A \otimes 1;\bullet \notag \\
	&=(1 + c);\pi_1 \otimes \pi_2;1 \otimes e_A \otimes 1;m_A \otimes 1;\bullet \notag \\
	&= (1+c);\pi_1 \otimes \pi_2;\bullet \notag
\end{align}

\end{prf}

We will need the following technical lemmas concerning the $T$-algebra $W(A,M)$.

\begin{lem}\label{4.5}
Let $(A,a)$ be a $T$-algebra, and let $M$ and $N$ be $A$-modules. Suppose $h\colon M\rarr N$ is an $A$-module map.
Then $A\oplus h\colon A\oplus M\rarr A\oplus N$ is a $T$-algebra map $W(A,M)\rarr W(A,N)$.
\end{lem}

\begin{prf}{} The result follows from the commutativity of the following two diagrams.

$$\bfig
\morphism(0,0)|a|/{>}/<1000,0>[T(A\oplus M)`T(A\oplus N);T(1_A \oplus h)]
\morphism(0,0)|l|/{>}/<0,-500>[`TA;T\pi_1]
\morphism(60,-500)|b|/{>}/<940,0>[`TA;1_{TA}]
\morphism(1000,0)|r|/{>}/<0,-440>[`;T\pi_1]
\morphism(0,-500)|l|/{>}/<0,-500>[`A;a]
\morphism(1000,-500)|r|/{>}/<0,-500>[`A;a]
\morphism(40,-1000)|b|/{>}/<920,0>[`;1_A]
\morphism(-10,-30)|l|/{@{>}@/_3em/}/<0,-930>[`;\beta_1^{AM}]
\morphism(1030,-30)|r|/{@{>}@/^3em/}/<0,-930>[`;\beta_1^{AN}]
\efig$$

$$\bfig
\morphism(0,0)|a|/{>}/<1700,0>[T(A\oplus M)`T(A\oplus N);T(1_A \oplus h)]
\morphism(0,0)|l|/{>}/<0,-600>[`T(A\oplus M) \otimes (A\oplus M);d]
\morphism(1700,0)|r|/{>}/<0,-600>[`T(A\oplus N) \otimes (A\oplus N);d]
\morphism(450,-600)|a|/{>}/<800,0>[`;T(1_A \oplus h) \otimes (1_A \oplus h)]
\morphism(0,-600)|l|/{>}/<0,-300>[`TA \otimes M;T\pi_1 \otimes \pi_2]
\morphism(1700,-600)|r|/{>}/<0,-300>[`TA \otimes N;T\pi_1 \otimes \pi_2]
\morphism(200,-900)|a|/{>}/<1300,0>[`;1_{TA} \otimes h]
\morphism(0,-900)|l|/{>}/<0,-300>[`A \otimes M;a \otimes 1_M]
\morphism(1700,-900)|r|/{>}/<0,-300>[`A \otimes N;a \otimes 1_N]
\morphism(150,-1200)|a|/{>}/<1400,0>[`;1_A \otimes h]
\morphism(0,-1200)|l|/{>}/<0,-300>[`M;\bullet_M]
\morphism(1700,-1200)|r|/{>}/<0,-300>[` N;\bullet_N]
\morphism(60,-1500)|b|/{>}/<1590,0>[`;h]
\morphism(-30,-50)|l|/{@{>}@/_6em/}/<0,-1400>[`;\beta_2^{AM}]
\morphism(1730,-50)|r|/{@{>}@/^6em/}/<0,-1400>[`;\beta_2^{AN}]
\efig$$

\end{prf}

%\begin{prf}{}

%$$\bfig
%\morphism(0,0)|a|/{>}/<700,0>[T(A\oplus M)`TA;T\pi_1]
%\morphism(0,0)|l|/{>}/<0,-300>[`A\oplus M;\beta]
%\morphism(700,0)|r|/{>}/<0,-300>[`A;a]
%\morphism(200,-300)|b|/{>}/<400,0>[`;\pi_1]
%\efig$$

%\end{prf}

The above calculations allow us to conclude:

\begin{prop}
Given a $T$-algebra $A$, the above construction defines a functor:

\[W(A,-)\colon A\mbox{-Mod}\longrightarrow \cC^T/A\]

Here, $\cC^T$ is the category of $T$-algebras and $\cC^T/A$ is the slice category over $A$. 

\end{prop}

It is the above series of observations that allows us to define a generalized notion of derivation depending on the given codifferential structure of $\cC$.

\begin{defn}{\em

\begin{itemize}
\item Let $(A,\nu)$ be a $T$-algebra. Let $M$ be an $A$-module. A {\em Beck $T$-derivation} for $A$ 
valued in $M$ is a $T$-algebra map

\[A -\!\!\!\!\!-\!\!\!\!\!-\!\!\!\longrightarrow W(A,M) \,\,\,\,\,\,\,\mbox{in $\cC^T/A$}\]

\noindent  in the slice category $\cC^T/A$.

\item A {\em $T$-derivation} is a morphism $\del\colon A\rarr M$ such that

\[\langle 1,\del\rangle\colon A-\!\!\!\!\!-\!\!\!\!\!-\!\!\!\longrightarrow A\oplus M\]

is a $T$-algebra homomorphism $A\rarr W(A,M)$.

\end{itemize}
}\end{defn}

\begin{rem}\label{SDer=Der}{\em 
Under the assumptions of Remark 3.1, suppose we are given $A \in \mathcal{C}^S$  where $S$ is the
symmetric algebra monad and $M \in A-Mod$. Then a morphism $\partial\colon A \to M$ in $\mathcal{C}$ is an $S$-derivation if and only if $\partial$ is a derivation.}
\end{rem}

\begin{rem} {\em Evidently, the two notions of {\it Beck $T$-derivation} and {\em $T$-derivation} are in bijective 
correspondence and we will use the two interchangeably. }
\end{rem}

We now give several equations for a map $\del\colon A\rarr M$ which are equivalent to $\del$ being a
$T$-derivation. 

\begin{prop} Let $(A,\nu)$ be a $T$-algebra, and let $M$ be an $A$-module. A morphism $\del\colon A\rarr M$ is a $T$-derivation
if and only if the following diagram commutes.

$$\bfig
\morphism(0,0)|a|/{>}/<700,0>[TA`T(A\oplus M);T\begin{pmatrix} 1_A \\ \partial \end{pmatrix}]
\morphism(0,0)|l|/{>}/<0,-300>[`A;\nu]
\morphism(700,0)|r|/{>}/<0,-300>[`M;\beta_2]
\morphism(20,-300)|b|/{>}/<660,0>[`;\partial]
\efig$$
\end{prop}

\begin{prf}{} 

Since $A\oplus M$ is a product, the requirement that $\langle 1_A,\del\rangle\colon A\rarr A\oplus M$
be a $T$-algebra homomorphism amounts to two equations, the second of which is expressed by the
above diagram whereas the first commutes by the following calculation

$$\bfig
\morphism(0,0)|a|/{>}/<700,0>[TA`T(A\oplus M);T\begin{pmatrix} 1_A \\ \partial \end{pmatrix}]
\morphism(0,0)|l|/{>}/<0,-550>[`A;\nu]
\morphism(700,0)|r|/{>}/<0,-300>[`TA;T\pi_1]
\morphism(700,-300)|r|/{>}/<0,-250>[`A;\nu]
\morphism(50,-40)|b|/{>}/<580,-250>[`;1_{TA}]
\morphism(50,-550)|b|/{>}/<600,0>[`;1_A]
\efig$$

\end{prf}

\begin{prop}\label{prop:chn_rule_charn_of_tderns} 
Let $(A,\nu)$ be a $T$-algebra, and let $M$ be an $A$-module. A morphism $\del\colon A\rarr M$ is a $T$-derivation
if and only if

$$\bfig
\morphism(0,0)|b|/{>}/<625,0>[TA \otimes A`A\otimes M;\nu \otimes \partial]
\morphism(750,0)|b|/{>}/<440,0>[`M;\bullet_M]
\morphism(0,300)|l|/{>}/<0,-250>[TA`;d]
\morphism(1190,300)|r|/{>}/<0,-250>[A`;\partial]
\morphism(70,300)|a|/{>}/<1090,0>[`;\nu]
\efig$$

\noindent commutes.

\end{prop}

\begin{prf}{} Calculate as follows:

$$\bfig
\morphism(0,0)|b|/{>}/<900,0>[T(A\oplus M)`T(A\oplus M) \otimes (A\oplus M);d]
\morphism(900,0)|b|/{>}/<1000,0>[T(A\oplus M) \otimes (A\oplus M)`TA \otimes M;T\pi_1 \otimes \pi_2]
\morphism(1900,0)|b|/{>}/<600,0>[TA \otimes M`A\otimes M;\nu \otimes 1_M]
\morphism(2500,0)|a|/{>}/<400,0>[A\otimes M`M;\bullet_M]
\morphism(0,600)|l|/{>}/<0,-600>[TA`T(A\oplus M);T\begin{pmatrix} 1_A \\ \partial \end{pmatrix}]
\morphism(0,600)|a|/{>}/<900,0>[TA`TA \otimes A;d]
\morphism(900,600)|l|/{>}/<0,-550>[TA \otimes A`;T\begin{pmatrix} 1_A \\ \partial \end{pmatrix} \otimes \begin{pmatrix} 1_A \\ \partial \end{pmatrix}]
\morphism(900,600)|l|/{>}/<1000,-600>[TA\otimes A`TA\otimes M;1_{TA} \otimes \partial]
\morphism(900,600)|r|/{>}/<1600,-600>[TA\otimes A`A\otimes M;\nu \otimes \partial]
\morphism(0,0)|b|/{@{>}@/_2.5em/}/<2900,0>[T(A\oplus M)`M;\beta_2]
\efig$$

Thus the result follows from the previous proposition.

\end{prf}

Whereas we have defined the notion of $T$-derivation in the setting of a given codifferential category, Theorem \ref{prop:chn_rule_charn_of_tderns} furnishes an equivalent definition that is applicable more generally, as follows.

\begin{defn}\label{def:generalized_defn_tdern}{\em
Let $\cC$ be a symmetric monoidal category equipped with an algebra modality $T$ and arbitrary morphisms $d_{TC}:TC \rightarrow TC \otimes C$ $(C \in \cC)$.  Given a $T$-algebra $A$ and an $A$-module $M$, a \textit{$T$-derivation} is a morphism $\del:A \rightarrow M$ such that the diagram of Proposition \ref{prop:chn_rule_charn_of_tderns} commutes.
}
\end{defn}

The new understanding of derivations captured by the above propositions allows us, among other things, to 
reexamine the definition of (co)differential categories, as seen by the following:

\smallskip

\begin{thm}\label{thm:chn_rule_iff_d_dern}
Let $\cC$ be a symmetric monoidal category equipped with an algebra modality $T$ and arbitrary morphisms $d_{TC}:TC \rightarrow TC \otimes C$ $(C \in \cC)$.  The Chain Rule equation for $d$ in the definition of 
codifferential category is equivalent to the statement that each component $d_{TC}$ is a $T$-derivation, where $TC \ox C$ is viewed as the 
free $TC$-module generated by $C$. 
\end{thm}

\begin{prf}{} 

$$\bfig
\morphism(0,0)|b|/{>}/<900,0>[T^2C \otimes TC`TC\otimes TC \otimes C;\mu \otimes d_{TC}]
\morphism(1200,0)|b|/{>}/<800,0>[`TC \otimes C;m_{TC}\otimes 1_C]
\morphism(0,300)|l|/{>}/<0,-240>[T^2C`;d_{T^2C}]
\morphism(2020,300)|r|/{>}/<0,-240>[TC`;d_{TC}]
\morphism(70,300)|a|/{>}/<1900,0>[`;\mu]
\efig$$

This equation is both the chain rule and the statement that $d_{TC}$ is a derivation.

\end{prf}

\subsection{Universal Beck $T$-derivations}

\begin{defn}{\em Given a $T$-algebra $A$, a {\em module of K\"ahler $T$-differentials} is an $A$-module, denoted 
$\Omega^T_A$, equipped with a universal $T$-derivation on $A$. This can be expressed in either of the following
two equivalent ways:

\begin{itemize}
\item A $T$-derivation $d\colon A\rarr \Omega^T_A$ such that for all $T$-derivations $\del\colon A\rarr M$,
there is a unique $A$-linear map $\hat{\del}\colon \Omega^T_A\rarr M$ such that $d;\hat{\del}=\del$.
\item A morphism $g\colon A\rarr W(A,\Omega^T_A)$ in $\cC^T/A$ such that for each map 
$\del\colon A\rarr W(A,M)$ in $\cC^T/A$, there is a unique $A$-linear homomorphism $\hat{\del}\colon\Omega^T_A\rarr M$
such that $g;W(A,\hat{\del})=\del$. 
\end{itemize}
}\end{defn}

We now explore the existence of universal derivations from this new $T$-perspective.

\begin{thm} Let \cC\ be a codifferential category, and let $C$ be an object of \cC. Then $d_{TC}\colon TC\rarr T(C)\ox C$ is a universal $T$-derivation.
\end{thm}

\begin{prf}{}

Since $d_{TC}$ satisfies the chain rule, it is a $T$-derivation. Since $T(C) \otimes C$ is the free $T(C)$-module on $C$, given any $T$-derivation $\partial:T(C) \to M$ there exists a unique $T(C)$-linear morphism $\partial^\#:TC \otimes C \to M$ such that $u_C^{T(C)};\partial^\# = \eta_C;\partial$. Hence by axiom (d3), the two morphisms from $C$ to $M$ in the following diagram are equal:

$$\bfig
\morphism(0,0)|a|/{>}/<600,0>[T(C)`T(C) \otimes C;d_{TC}]
\morphism(600,0)|r|/{>}/<0,-400>[T(C) \otimes C`M;\partial^\#]
\morphism(0,0)|b|/{>}/<600,-400>[T(C)`M;\partial]
\morphism(-500,0)|a|/{>}/<500,0>[C`T(C);\eta_C]
\morphism(-500,0)/{}/<0,0>[C`;]
\efig$$

Equivalently,

$$\bfig
\morphism(0,0)|a|/{>}/<800,0>[TC`W(TC,TC \otimes C);\begin{pmatrix} 1 \\ d_{TC} \end{pmatrix}]
\morphism(800,0)|r|/{>}/<0,-400>[W(TC,TC\otimes C)`W(TC,M);W(TC,\partial^\#)]
\morphism(0,0)|b|/{>}/<800,-400>[TC`W(TC,M);\begin{pmatrix} 1 \\ \partial \end{pmatrix}]
\efig$$

\bigskip

\noindent commutes when preceded by $\eta_{C}$. Since this is a diagram of $T$-algebra homomorphisms, it commutes if and only if it commutes when preceded by $\eta_C$

\end{prf}

We now address the issue of extending the existence of universal $T$-derivations to arbitrary $T$-algebras. 

\begin{prop}\label{4.16} Let $(A,a)$ and $(B,b)$ be $T$-algebras and $M$ a $B$-module. Let $g\colon A\rarr B$ be a $T$-algebra homomorphism.
Then $g\oplus M\colon A\oplus M\rarr B\oplus M$ is a map of $T$-algebras $W(A,M_A)\rarr W(B,M)$, where $M_A$ is $M$ with evident induced action of $A$.
\end{prop}

\begin{prf}{}
The result follows from the commutativity of the following two diagrams.

$$\bfig
\morphism(0,0)|a|/{>}/<1000,0>[T(A\oplus M)`T(B\oplus M);T(g \oplus 1_M)]
\morphism(0,0)|l|/{>}/<0,-500>[`TA;T\pi_1]
\morphism(60,-500)|b|/{>}/<940,0>[`TB;Tg]
\morphism(1000,0)|r|/{>}/<0,-440>[`;T\pi_1]
\morphism(0,-500)|l|/{>}/<0,-500>[`A;a]
\morphism(1000,-500)|r|/{>}/<0,-500>[`B;b]
\morphism(40,-1000)|b|/{>}/<920,0>[`;g]
\morphism(-10,-30)|l|/{@{>}@/_3em/}/<0,-930>[`;\beta_1^{AM}]
\morphism(1030,-30)|r|/{@{>}@/^3em/}/<0,-920>[`;\beta_1^{BM}]
\efig$$

$$\bfig
\morphism(0,0)|a|/{>}/<1700,0>[T(A\oplus M)`T(B\oplus M);T(g \oplus 1_M)]
\morphism(0,0)|l|/{>}/<0,-600>[`T(A\oplus M) \otimes (A\oplus M);d]
\morphism(1700,0)|r|/{>}/<0,-600>[`T(B\oplus M) \otimes (B\oplus M);d]
\morphism(450,-600)|a|/{>}/<800,0>[`;T(g \oplus 1_M) \otimes (g \oplus 1_M)]
\morphism(0,-600)|l|/{>}/<0,-300>[`TA \otimes M;T\pi_1 \otimes \pi_2]
\morphism(1700,-600)|r|/{>}/<0,-300>[`TB \otimes M;T\pi_1 \otimes \pi_2]
\morphism(200,-900)|a|/{>}/<1300,0>[`;Tg \otimes 1_M]
\morphism(0,-900)|l|/{>}/<0,-300>[`A \otimes M;a \otimes 1_M]
\morphism(1700,-900)|r|/{>}/<0,-300>[`B \otimes M;b \otimes 1_M]
\morphism(150,-1200)|a|/{>}/<1400,0>[`;g \otimes 1_M]
\morphism(0,-1200)|l|/{>}/<0,-300>[`M;\bullet_M]
\morphism(1700,-1200)|r|/{>}/<0,-300>[` M;\bullet_M]
\morphism(60,-1500)|b|/{>}/<1590,0>[`;1_M]
\morphism(-30,-50)|l|/{@{>}@/_6em/}/<0,-1400>[`;\beta_2^{AM}]
\morphism(1730,-50)|r|/{@{>}@/^6em/}/<0,-1400>[`;\beta_2^{BM}]
\efig$$

\end{prf}

\begin{prop}\label{page8}
With assumptions as in previous proposition, let $\partial\colon A\rarr M$ be such that $\langle g,\partial\rangle\colon A\rarr W(B,M)$ is a map of 
$T$-algebras.  Then $\partial\colon A\rarr M_A$ is a $T$-derivation.
\end{prop}

\begin{prf}{}
This follows from the following calculation, which uses that $g\oplus 1_M$ is a $T$-algebra homomorphism
by the previous proposition. 

$$\bfig
\morphism(0,0)|a|/{>}/<1600,0>[TA`T(A\oplus M);T\begin{pmatrix} 1_A \\ \partial \end{pmatrix}]
\morphism(0,0)|l|/{>}/<0,-900>[TA`A;a]
\morphism(1600,0)|r|/{>}/<0,-900>[T(A\oplus M)`M;\beta_2^{AM}]
\morphism(0,-900)|b|/{>}/<1600,0>[A`M;\partial]
\morphism(0,0)|b|/{>}/<800,-300>[TA`T(B\oplus M);T\begin{pmatrix} g \\ \partial \end{pmatrix}]
\morphism(1600,0)|b|/{>}/<-800,-300>[T(A\oplus M)`T(B\oplus M);T(g\oplus 1_M)]
\morphism(800,-300)|r|/{>}/<0,-300>[T(B\oplus M)`M;\beta_2^{BM}]
\morphism(0,-900)|a|/{>}/<800,300>[A`M;\partial]
\morphism(800,-600)|a|/{>}/<800,-300>[M`M;1_M]
\efig$$

\end{prf}

\begin{defn}{\em
Let $Alg$ be the category of commutative algebras in a codifferential category $\cC$ and let 
$(-)-Mod\colon Alg^{op}\rarr Cat$
be the usual functor associating to an algebra its category of representations. The functor acts 
on morphisms by the usual restriction of scalars.

Composing with the functor $F^{op}\colon (\cC^T)^{op}\rarr Alg^{op}$ we obtain a functor $H\colon{\cC^T}^{op}\rarr Cat$. When we apply the usual Grothendieck construction to this functor, we obtain a category fibred over $\cC^T$ which we call $Mod_T$. Objects 
are pairs $(A,M)$ with $A$ a $T$-algebra and $M$ an $A$-module. Arrows are pairs $(g,h)\colon (A,M)\rarr (B,N)$ with $g\colon A\rarr B$
a $T$-algebra map and $h\colon M\rarr N_A$ a map of $A$-modules. Here $N_A$ is the restriction of scalars of $N$
along $g$ (Remark \ref{rem:derns_alg_maps_mon}).

}\end{defn}

\begin{thm}
There is a functor $W\colon Mod_T\rarr(\mathcal{C}^T)^{\rightarrow}$ that makes the following diagram commute:

$$\bfig
\morphism(0,0)|a|/{>}/<750,0>[Mod_T`(\mathcal{C}^T)^{\rightarrow};W]
\morphism(0,0)|b|/{>}/<375,-375>[Mod_T`\mathcal{C}^T;]
\morphism(750,0)|b|/{>}/<-375,-375>[(\mathcal{C}^T)^{\rightarrow}`\mathcal{C}^T;cod]
\efig$$

The functor is defined by:

\[\mbox{On objects: \,\,\,}(A,M)\mapsto [W(A,M)\stackrel{\pi_1}{-\!\!\!\!\!-\!\!\!\!\!-\!\!\!\longrightarrow}A] \]
\[\mbox{On arrows:\,\,\,}(A,M)\stackrel{(g,h)}{-\!\!\!\!\!-\!\!\!\!\!-\!\!\!\longrightarrow}(B,N)\mapsto\mbox{the following:}\]

$$\bfig
\morphism(0,0)|a|/{>}/<1000,0>[W(A,M)`W(B,N);W(h,g) := g\oplus h]
\morphism(0,0)|l|/{>}/<0,-450>[W(A,M)`A;\pi_1]
\morphism(1000,0)|r|/{>}/<0,-450>[W(B,N)`B;\pi_1]
\morphism(0,-450)|b|/{>}/<1000,0>[A`B;g]

\efig$$

\noindent This functor is fibred over the base category $\cC^T$.

\end{thm}

\begin{prf}{}
We evidently have that $(1\oplus h);(g\oplus 1)=g\oplus h$ is a map of $T$-algebras by Lemma \ref{4.5} and Proposition \ref{4.16}, and so we have a functor making the triangle commute. 

Now given a $T$-algebra homomorphism $g\colon A\rarr B$ and a $B$-module $N$, we get a cartesian arrow 
over $g$ in $Mod_T$ as $(g,1_N)\colon (A,N_A)\rarr(B,N)$. It suffices to show that 

\smallskip

$$\bfig
\morphism(0,0)|a|/{>}/<1000,0>[W(A,N_A)`W(B,N);W(g,1_N)]
\morphism(0,0)|l|/{>}/<0,-500>[W(A,N_A)`A;\pi_1]
\morphism(1000,0)|r|/{>}/<0,-500>[W(B,N)`B;\pi_1]
\morphism(0,-500)|b|/{>}/<1000,0>[A`B;g]
\efig$$

\smallskip

\noindent is a pullback.  Given $f\colon Q\rarr A$ and $q\colon Q\rarr W(B,N)$ in $\mathcal{C}^T$ such that $f;g=q;\pi_1$, we find that $q=\langle f;g,\partial\rangle$ for some $\partial\colon Q\rarr N$. By Lemma \ref{page8}, we conclude $\partial\colon Q\rarr N_Q$ is a $T$-derivation. So $\langle 1_Q,\partial\rangle$ 
is a $T$-algebra map and thus $\langle 1_Q,\partial\rangle;f\oplus 1=\langle f,\del\rangle\colon 
Q\rarr W(A,N_A)$ is a $T$-algebra map. The result now follows.
\end{prf}

\begin{defn}{\em Let $A$ be a $T$-algebra and $(B,M)$ in $Mod_T$. Let $Der(A,(B,M))$ be the set of all pairs
 $(g,\del)$ with $g\colon A\rarr B$ a $T$-algebra map and $\del\colon A\rarr M_A$ a $T$-derivation.
}\end{defn}

We now record two related results which are straightforward. 

\begin{prop}
The operation $Der$ of the previous definition is functorial in both variables and forms part of a natural isomorphism:

\[\cC^T(A,W(B,M))\cong Der(A,(B,M))\]
\end{prop}

This result extends to the slice category in a straightforward way.

\begin{prop} Given a $T$-algebra map $g\colon A\rarr B$, we have the following natural isomorphism:

\[\cC^T/B(A,W(B,M))\cong Der(A,M_A)\]
\end{prop}

We now present the main result of the section, demonstrating that the construction of K\"ahler modules for $T$-algebras lifts to 
the setting of $T$-derivations.

\begin{thm}\label{main} Suppose $\mathcal{C}$ has reflexive coequalizers, and that these are preserved by $\otimes$ in each variable. Then every $T$-algebra $(A,\nu)$ has a universal $T$-derivation.
\end{thm}

\begin{prf}{}
Let $g:A \to B$ be a morphism of $T$-algebras, and suppose that universal $T$-derivations $d_A:A \to \Omega_A^T$, $d_B:B \to \Omega_B^T$ exist. Then there is a unique $A$-linear morphism $\Omega_g^T$ such that 
$$\bfig
\morphism(0,0)|b|/{>}/<400,0>[A`B;g]
\morphism(0,0)|l|/{>}/<0,300>[A`\Omega_A^T;d_A]
\morphism(400,0)|r|/{>}/<0,300>[B`\Omega_B^T;d_B]
\morphism(0,300)|a|/{>}/<400,0>[\Omega_A^T`\Omega_B^T;\Omega_g^T]
\efig$$
commutes, where $\Omega_B^T$ is considered as an $A$-module by restriction of scalars along $g$. This follows from the observation that $g;d_B:A \to \Omega_B^T$ is a $T$-derivation.

\begin{lem}
Suppose we are given morphisms in the category $Alg$ as follows which constitute a reflexive coequalizer in \cC
$$\bfig
\morphism(0,0)|a|/@{>}@<+2pt>/<500,0>[A_1`A_2;f]
\morphism(0,0)|b|/@{>}@<-2pt>/<500,0>[A_1`A_2;g]
\morphism(500,0)|a|/{>}/<500,0>[A_2`A_3;k]
\morphism(1000,0)/{}/<0,0>[A_3`;]
\efig$$

Let $M_i$ be an $A_i$-module for $i = 1,2$, and let $\phi:M_1\to f^*(M_2)$ and $\gamma:M_1 \to g^*(M_2)$ be $A_1$-linear, where $f^*(M_2)$ and $g^*(M_2)$ denote $M_2$ equipped with the $A_1$-module structures induced by f and g, respectively. Suppose
$$\bfig
\morphism(0,0)|a|/@{>}@<+2pt>/<500,0>[M_1`M_2;\phi]
\morphism(0,0)|b|/@{>}@<-2pt>/<500,0>[M_1`M_2;\gamma]
\morphism(500,0)|a|/{>}/<500,0>[M_2`M_3;\kappa]
\morphism(1000,0)/{}/<0,0>[M_3`;]
\efig$$
is a reflexive coequalizer in $\mathcal{C}$. Then there is a unique $A_3$-module structure on $M_3$ such that $\kappa:M_2 \to k^*(M_3)$ is $A_2$-linear.
\end{lem}

\begin{proof}
Since $\otimes$ preserves reflexive coequalizers, the rows and columns of the following diagram are reflexive coequalizers:
$$\bfig
	%First Horizontal
	\morphism(0,0)|a|/@{>}@<+2pt>/<1000,0>[A_1 \otimes M_1`A_1 \otimes M_2;1 \otimes \phi]
	\morphism(0,0)|b|/@{>}@<-2pt>/<1000,0>[A_1 \otimes M_1`A_1 \otimes M_2;1\otimes \gamma]
	\morphism(1000,0)|a|/{>}/<1000,0>[A_1 \otimes M_2`A_1 \otimes M_3;1 \otimes \kappa]
	%First Horizontal

	%First Vertical
	\morphism(0,0)|l|/@{>}@<-2pt>/<0,-750>[A_1 \otimes M_1`A_2 \otimes M_1;g \otimes 1]
	\morphism(0,0)|r|/@{>}@<+2pt>/<0,-750>[A_1 \otimes M_1`A_2 \otimes M_1;f \otimes 1]
	\morphism(0,-750)|l|/{>}/<0,-750>[A_2 \otimes M_1`A_3 \otimes M_1;k \otimes 1]
	%First Vertical

	%Second Horizontal
	\morphism(0,-750)|a|/@{>}@<+2pt>/<1000,0>[A_2 \otimes M_1`A_2 \otimes M_2;1 \otimes \phi]
	\morphism(0,-750)|b|/@{>}@<-2pt>/<1000,0>[A_2 \otimes M_1`A_2 \otimes M_2;1 \otimes \gamma]
	\morphism(1000,-750)|a|/{>}/<1000,0>[A_2 \otimes M_2`A_2 \otimes M_3;1 \otimes \kappa]
	%Second Horizontal

	%Second Vertical
	\morphism(1000,0)|l|/@{>}@<-2pt>/<0,-750>[A_1 \otimes M_2`A_2 \otimes M_2;g \otimes 1]
	\morphism(1000,0)|r|/@{>}@<+2pt>/<0,-750>[A_1 \otimes M_2`A_2 \otimes M_1;f \otimes 1]
	\morphism(1000,-750)|l|/{>}/<0,-750>[A_2 \otimes M_2`A_3 \otimes M_2;k \otimes 1]
	%Second Vertical
	
	%Third Horizontal
	\morphism(0,-1500)|a|/@{>}@<+2pt>/<1000,0>[A_3 \otimes M_1`A_3 \otimes M_2;1 \otimes \phi]
	\morphism(0,-1500)|b|/@{>}@<-2pt>/<1000,0>[A_3 \otimes M_1`A_3 \otimes M_2;1 \otimes \gamma]
	\morphism(1000,-1500)|a|/{>}/<1000,0>[A_3 \otimes M_2`A_3 \otimes M_3;1 \otimes \kappa]
	%Third Horizontal

	%Third Vertical
	\morphism(2000,0)|l|/@{>}@<-2pt>/<0,-750>[A_1 \otimes M_3`A_2 \otimes M_3;g \otimes 1]
	\morphism(2000,0)|r|/@{>}@<+2pt>/<0,-750>[A_1 \otimes M_3`A_2 \otimes M_3;f \otimes 1]
	\morphism(2000,-750)|l|/{>}/<0,-750>[A_2 \otimes M_3`A_3 \otimes M_3;k \otimes 1]
	\morphism(2000,-1500)/{}/<0,0>[A_3 \otimes M_3`;]
	%Third Vertical
\efig$$
By Johnstone's lemma, Lemma 0.17, p. 4 \cite{J}, it follows that the top row of
$$\bfig
	\morphism(0,0)|a|/@{>}@<+2pt>/<1000,0>[A_1 \otimes M_1`A_2 \otimes M_2;f \otimes \phi]
	\morphism(0,0)|b|/@{>}@<-2pt>/<1000,0>[A_1 \otimes M_1`A_2 \otimes M_2;g \otimes \gamma]
	\morphism(1000,0)|a|/{>}/<1000,0>[A_2 \otimes M_2`A_3 \otimes M_3;k \otimes \kappa]

	\morphism(0,-500)|a|/@{>}@<+2pt>/<1000,0>[M_1`M_2;\phi]
	\morphism(0,-500)|b|/@{>}@<-2pt>/<1000,0>[M_1`M_2;\gamma]
	\morphism(1000,-500)|a|/{>}/<1000,0>[M_2`M_3;\kappa]

	\morphism(0,0)|l|/{>}/<0,-500>[A_1 \otimes M_1`M_1;\bullet_1]
	\morphism(1000,0)|l|/{>}/<0,-500>[A_2 \otimes M_2`M_2;\bullet_2]
	\morphism(2000,0)|l|/{-->}/<0,-500>[A_3 \otimes M_3`M_3;\bullet_3]
	\morphism(2000,0)/{}/<0,0>[A_3 \otimes M_3`;]
	\morphism(2000,-500)/{}/<0,0>[M_3`;]
\efig$$
is also a reflexive coequalizer. We have that
\begin{align}
	f\otimes \phi;\bullet_2;\kappa &= 1_{A_1} \otimes \phi;f\otimes 1_{M_2};\bullet_2;\kappa \notag \\
				                 &= \bullet_1;\phi;\kappa \notag \\
					      &= \bullet_1;\gamma;\kappa \notag \\
					      &= 1_{A_1} \otimes \gamma; g \otimes 1_{M_2}; \bullet_2;\kappa \notag \\
					      &= g\otimes \gamma;\bullet_2;\kappa \notag
\end{align}
It follows that $\bullet_3:A_3 \otimes M_3 \to M_3$ is constructed as the unique map making the right-hand square in the above diagram commute. Hence it suffices to show that $\bullet_3$ is an $A_3$-module structure map on $M_3$. Again using Johnstone's Lemma, the top row of the following diagram is a reflexive coequalizer
$$\bfig
	\morphism(0,0)|a|/@{>}@<+2pt>/<1000,0>[A_1 \otimes A_1 \otimes M_1`A_2 \otimes A_2 \otimes M_2;f\otimes f \otimes \phi]
	\morphism(0,0)|b|/@{>}@<-2pt>/<1000,0>[A_1 \otimes A_1 \otimes M_1`A_2 \otimes A_2 \otimes M_2;g\otimes g \otimes \gamma]
	\morphism(1000,0)|a|/{>}/<1000,0>[A_2 \otimes A_2 \otimes M_2`A_3\otimes A_3 \otimes M_3;k\otimes k \otimes \kappa]

	\morphism(0,-500)|a|/@{>}@<+2pt>/<1000,0>[M_1`M_2;\phi]
	\morphism(0,-500)|b|/@{>}@<-2pt>/<1000,0>[M_1`M_2;\gamma]
	\morphism(1000,-500)|a|/{>}/<1000,0>[M_2`M_3;\kappa]

	\morphism(0,0)|l|/{>}/<0,-500>[A_1\otimes A_1 \otimes M_1`M_1;1\otimes \bullet_1;\bullet_1 = m_{A_1}\otimes 1;\bullet_1]
	\morphism(1000,0)|l|/{>}/<0,-500>[A_2\otimes A_2 \otimes M_2`M_2;1\otimes \bullet_2;\bullet_2 = m_{A_2}\otimes 1;\bullet_2]
	\morphism(2000,0)|l|/{-->}/<0,-500>[A_3 \otimes A_3 \otimes M_3`M_3;]
	\morphism(2000,0)/{}/<0,0>[A_3 \otimes A_3 \otimes M_3`;]
	\morphism(2000,-500)/{}/<0,0>[M_3`;]
\efig$$
It follows that there is a unique map $A_3\otimes A_3 \otimes M_3 \to M_3$ making the right-hand square commute. Since both $1_{A_3}\otimes \bullet_3;\bullet_3$ and $m_{A_3}\otimes 1_{M_3};\bullet_3$ satisfy this, the result follows. 
\end{proof}
Continuing with the proof of our theorem, since $\mu_A$ and $T\nu$ are $T$-algebra morphisms, they induce maps $\Omega_{\mu}^T$ and $\Omega_{T\nu}^T$ from $\Omega_{T^2A}^T$ to $\Omega_{TA}^T$, which exist by Theorem 4.14. Furthermore, there exists a map $\Omega_{T\eta}^T$ induced by $T\eta$, which splits both of these maps. Consider the following diagram. We define $d_{A}$ as the unique morphism in \cC\ such that 
$\nu;d_A=d_{TA};\Omega_\nu^T$, which exists since $\nu$ is the coequalizer of $\mu$ and $T\nu$.
Here we take $\Omega_\nu^T:\Omega_{TA}^T \to \Omega_{A}^T$ to be the coequalizer. 

$$\bfig
	\morphism(0,0)|a|/@{>}@<+2pt>/<1000,0>[\Omega_{T^2A}^T`\Omega_{TA}^T;\Omega_{\mu}^T]
	\morphism(0,0)|b|/@{>}@<-2pt>/<1000,0>[\Omega_{T^2A}^T`\Omega_{TA}^T;\Omega_{T\nu}^T]
	\morphism(1000,0)|a|/{>}/<1000,0>[\Omega_{TA}^T`\Omega_{A}^T;\Omega_{\nu}^T]

	\morphism(0,-500)|a|/@{>}@<+2pt>/<1000,0>[T^2A`TA;\mu]
	\morphism(0,-500)|b|/@{>}@<-2pt>/<1000,0>[T^2A`TA;T\nu]
	\morphism(1000,-500)|a|/{>}/<1000,0>[TA`A;\nu]

	\morphism(0,-500)|l|/{>}/<0,500>[T^2A`\Omega_{T^2A}^T;d_{T^2A}]
	\morphism(1000,-500)|l|/{>}/<0,500>[TA`\Omega_{TA}^T;d_{TA}]
	\morphism(2000,-500)|l|/{-->}/<0,500>[A`\Omega_{A}^T;d_A]
	\morphism(2000,-500)/{}/<0,0>[A`;]
	\morphism(2000,0)/{}/<0,0>[\Omega_A^T`;]
\efig$$

One readily verifies that the preceding lemma applies so that  
$\Omega_A^T$ is equipped with an $A$-module structure, which makes $\Omega_\nu^T$ $TA$-linear. We find that $d_A = \eta_A;d_{TA};\Omega_\nu^T$ since $\nu;\eta_A;d_{TA};\Omega_\nu^T=d_{TA};\Omega_\nu^T=\nu;d_A$, where the first equation is established through a short computation using the fact that $\nu;\eta_A=\eta_{TA};T\nu$.

Since
$$\bfig
	\morphism(0,0)|a|/{>}/<1000,0>[TA`A;\nu]
	\morphism(0,0)|l|/{>}/<0,-500>[TA`W(TA,\Omega_{TA}^T);\begin{pmatrix} 1_{TA} \\ d_{TA} \end{pmatrix}]
	\morphism(0,-500)|b|/{>}/<1000,0>[W(TA,\Omega_{TA}^T)`W(A,\Omega_A^T);W(\nu,\Omega_\nu^T)]
	\morphism(1000,0)|r|/{>}/<0,-500>[A`W(A,\Omega_A^T);\begin{pmatrix} 1_A \\ d_A \end{pmatrix}]
\efig$$
commutes, it follows that the right-hand map is a $T$-algebra homomorphism and therefore that $d_A$ is a $T$-derivation. Indeed, the counterclockwise composite is evidently a $T$-algebra homomorphism, and since $\nu$ is a $T$-algebra homomorphism that is split epi in $\mathcal{C}$, the fact that the right-hand map is a $T$-algebra homomorphism follows readily. 

Now suppose that $\partial:A \to M$ is a $T$-derivation. Then $\nu;\partial$ is a $T$-derivation, which must factor through $d_{TA}$ via a morphism of $TA$-modules $\partial '$. Since
\begin{align}
	d_{T^2A};\Omega_{T\nu}^T;\partial ' &= T\nu;d_{TA};\partial ' \notag \\
						     &= T\nu;\nu;\partial \notag \\
						     &= \mu;\nu;\partial \notag \\
						     &= \mu;d_{TA};\partial ' \notag \\
						     &= d_{T^2A};\Omega_\mu^T;\partial ' \notag
\end{align}
it follows from the universal property of $d_{T^2A}$ that $\Omega_{T\nu}^T;\partial ' = \Omega_{\mu}^T;\partial '$, so that $\partial '$ factors uniquely through $\Omega_\nu^T$ via a map $\partial^{\#}:\Omega_A^T \to M$. Since $\nu \otimes \Omega_\nu^T$ is a coequalizer, the following computation shows that this map is $A$-linear:
\begin{align}
	\nu\otimes\Omega_\nu^T;\bullet_A;\partial^{\#} &= \bullet_{TA};\Omega_{\nu}^T;\partial^{\#} \notag \\
								  &= \bullet_{TA};\partial ' \notag \\
								  &= 1_{TA} \otimes \partial ';\nu \otimes 1_A;\bullet_A \notag \\
								  &= \nu \otimes\Omega_\nu^T;1_A\otimes\partial^{\#};\bullet_A \notag 
\end{align}
Finally, we show that $\partial^{\#}$ is the unique $A$-linear morphism, which makes
$$\bfig
	\morphism(0,0)|a|/{>}/<450,0>[A`\Omega_{A};d_A]
	\morphism(0,0)|b|/{>}/<450,-450>[A`M;\partial]
	\morphism(450,0)|r|/{>}/<0,-450>[\Omega_{A}`M;\partial^{\#}]
\efig$$
commute. First, observe that $\nu;d_A;\partial^{\#} = d_{TA};\Omega_\nu^T;\partial^{\#} = d_{TA};\partial ' = \nu;\partial$ so that this does indeed commute after cancellation of $\nu$. Now suppose that there exists another $A$-linear map $k:\Omega_A^T \to M$ such that $d_A;k = \partial$. Then
\begin{align}
	d_{TA};\Omega_\nu^T;k &= \nu;d_A;k  \notag \\
		         		     &= \nu;\partial \notag \\
				     &= d_{TA};\partial ' \notag \\
				     &= d_{TA};\Omega_\nu^T;\partial^{\#} \notag
\end{align}
The universal property of $d_{TA}$ dictates that $\Omega_\nu^T;k = \Omega_\nu^T;\partial^{\#}$ and therefore $k = \partial^{\#}$ and the proof is complete.

\end{prf}

\section{An alternative definition of (co)differential category}

Realization of the importance of the symmetric algebra in the analysis of K\"ahler categories also has the benefit that
it leads to a succinct alternative definition of codifferential category as follows.

\begin{thm}\label{thm:alt_charn_codiff_cats}
Let $\cC$ be an additive symmetric monoidal category for which the symmetric algebra monad $\SSS$ on $\cC$ exists.  Assume that $\cC$ has reflexive coequalizers and that these are preserved by the tensor product in each variable.  Then to equip $\cC$ with the structure of a codifferential category is, equivalently, to equip $\cC$ with
\begin{enumerate}
\item a monad $T$,
\item a monad morphism $\lambda:\SSS \rightarrow T$, and
\item a transformation $d_{TC}:TC \rightarrow TC \otimes C$ natural in $C \in \cC$
\end{enumerate}
such that
\begin{enumerate}
\item[(a)] the diagram
$$
\xymatrix{
SC \ar[d]_{d_{SC}} \ar[r]^{\lambda_C} & TC \ar[d]^{d_{TC}} \\
SC \otimes C \ar[r]^{\lambda_C \otimes 1_C} & TC \otimes C
}
$$
commutes for each $C \in \cC$, where $d_{SC}$ is the deriving transformation carried by $\SSS$, and
\item[(b)] the Chain Rule axiom of Definition \ref{def:codiff_cat} holds, i.e. each $d_{TC}$ is a $T$-derivation.
\end{enumerate}
\end{thm}
\begin{prf}{}

By Remark \ref{monadic},  the category of commutative algebras in $\cC$ is monadic over $\cC$ and so can be identified with the category of $\SSS$-algebras.  By Theorem \ref{thm:alg_modality_via_mnd_mor}, we know that algebra modalities on $\cC$ are in bijective correspondence with pairs $(T,\lambda)$ consisting of a monad $T$ on $\cC$ and a monad morphism $\lambda\colon\SSS \rightarrow T$.  Suppose we are given such a pair $(T,\lambda)$, together with a natural tranformation $d_{T(-)}$ satisfying (a) and (b).

\begin{quotation}
\noindent\textit{Claim:} Any $T$-derivation $\partial:A \rightarrow M$ is, in particular, an $\SSS$-derivation, equivalently by Remark \ref{SDer=Der}, a derivation in the ordinary sense (Definition \ref{def:ordinary_dern}).
\end{quotation}

To prove this claim, observe that the following diagram commutes, by (a) and Definition
\ref{def:generalized_defn_tdern}, where $a$ is the given $T$-algebra structure on $A$.
$$
\xymatrix{
SA \ar[d]_{d_{SA}} \ar[r]^{\lambda_A} & TA \ar[d]^{d_{TA}} \ar[rr]^a & & A \ar[d]^\partial \\
SA \otimes A \ar[r]_{\lambda_A \otimes 1} & TA \otimes A \ar[r]_{a \otimes \partial} & A \otimes M \ar[r]_\bullet & M
}
$$
But the upper row is the $\SSS$-algebra structure acquired by $A$ via Theorem \ref{thm:alg_modality_via_mnd_mor}, so by Definition \ref{def:generalized_defn_tdern} the Claim is proved.

We have assumed that $d$ satisfies the Chain Rule axiom, equivalently that each component $d_{TC}:TC \rightarrow TC \otimes C$ is a $T$-derivation (Theorem \ref{thm:chn_rule_iff_d_dern}), so by the Claim, $d_{TC}$ is an $\SSS$-derivation, equivalently, an ordinary derivation.  Hence the axioms (d1) and (d2) of Definition \ref{def:codiff_cat} hold, since together they assert exactly that each component $d_{TC}$ is an ordinary derivation.  We also know that axiom (d4) (the Chain Rule) holds, by assumption (b), so it suffices to prove that (d3) holds.  Indeed, (d3) asserts that the periphery of the following diagram commutes
\begin{equation}\label{eqn:linearity_diag}
\xymatrix{
C \ar[d]_(.4){\wr} \ar@/^3ex/[rr]^{\eta^T_C} \ar[r]_{\eta^\SSS_C} & SC \ar[d]^{d_{SC}} \ar[r]_{\lambda_C} & TC \ar[d]^{d_{TC}}\\
k\otimes C \ar[r]^{e \otimes 1} \ar@/_3ex/[rr]_{e \otimes 1} & SC \otimes C \ar[r]^{\lambda_C \otimes 1} & TC \otimes C
}
\end{equation}
where $\eta^T$ and $\eta^\SSS$ are the units of $T$ and $\SSS$, respectively.  The upper cell commutes since $\lambda$ is a monad morphism, and the lower cell commutes since $\lambda_C$ is an \SSS-homomorphism, i.e a homomorphism of algebras.  The leftmost cell commutes since $\cC$ is a codifferential category when equipped with $\SSS$ (\ref{thm:smalg_mnd_codiff}), and the rightmost cell commutes by (a).

Conversely, let us instead assume that $(\cC,T,d)$ is a codifferential category.  Then since axiom (d3) holds, the 
periphery of the diagram \eqref{eqn:linearity_diag} commutes, but we also know that the upper, lower, and leftmost 
cells in \eqref{eqn:linearity_diag} commute.  Hence, whereas our aim is to show that (a) holds, i.e., that the rightmost 
square in \eqref{eqn:linearity_diag} commutes, we know that this square `commutes when preceded by $\eta^
\SSS_C$'.  But by axioms (d1) and (d2), $d_{TC}$ is an ordinary derivation, equivalently, an $\SSS$-derivation 
(\ref{def:ordinary_dern}), so the composite $\lambda_C;d_{TC}$ is an $\SSS$-derivation since $\lambda_C$ is an 
algebra map.  Also, $d_{SC}$ is an $\SSS$-derivation, and one readily checks that $\lambda_C 
\otimes 1:SC \otimes C \rightarrow TC \otimes C$ is a morphism of $SC$-modules (where $TC \otimes C$ carries the 
$SC$-module structure that it acquires by restriction of scalars along the algebra homomorphism $\lambda_C$).  
Hence the composite $d_{SC};\lambda_C \otimes 1$ is an $\SSS$-derivation.  Therefore both composites in the square 
in question are $\SSS$-derivations and so are uniquely determined by their composites with $\eta^\SSS_C:C 
\rightarrow SC$, which are equal.
\end{prf}

An advantage of this definition is that it immediately paves the way for variations of the theory of differential 
categories and differential linear logic. For example, to obtain noncommutative variants, one can replace the 
symmetric algebra in the above construction with a different endofunctor.


\begin{thebibliography}{12}

\bibitem{B} J. Beck. {\it Triples, Algebras and Cohomology}. Thesis, (1967). 
{\it can be found in:} Reprints in Theory and Applications of Categories. 
\bibitem{BCPS} R. Blute, R. Cockett, T. Porter, R. Seely. K\"ahler categories,
		Cahiers de Topologie et Geometrie Differentielle 52, pp. 253-268, (2011).
\bibitem{BCS} R. Blute, R. Cockett, R. Seely. Differential categories,
		Mathematical Structures in Computer Science 16, pp. 1049-1083, (2006).
\bibitem{BCS2} R. Blute, R, Cockett, R.A.G. Seely. Cartesian differential categories, Theory and
Applications of Categories 22, pp. 622-672, (2009).
\bibitem{BET} R. Blute, T. Ehrhard, C. Tasson.
A convenient differential category, Cahiers de Topologie et Geometrie Differentielle 52, pp. 211-232, (2012).
\bibitem{Bour} N. Bourbaki. {\it Algebra Volumes I-III}. Springer-Verlag, (1974). 
\bibitem{Co} R. Cockett. Lectures on noncommutative K\"ahler categories. Preprint (2014). 
\bibitem{DK} E. J. Dubuc, A. Kock. On 1-form classifiers. Communications in Algebra 12, pp. 1471-1531, (1984). 
\bibitem{Ehr1} T. Ehrhard. On K\"othe sequence spaces and
linear logic.  Mathematical Structures in Computer
Science 12, pp. 579-623, (2002).
\bibitem{Ehr2} T. Ehrhard. Finiteness spaces.
Mathematical Structures in Computer Science 15,
pp. 615-646, (2005)
\bibitem{ER1}   T.   Ehrhard,  L.  Regnier.   The   differential
  $\lambda$-calculus.   Theoretical Computer  Science  309, pp.  1-41,
  (2003).
\bibitem{ER2} T.  Ehrhard, L.  Regnier.  Differential interaction
  nets.  Theoretical Computer Science 364, pp. 166--195, (2006).
\bibitem{E} D. Eisenbud. \emph{Commutative Algebra with a View Toward Algebraic Geometry},
Springer-Verlag, (2004).
\bibitem{G} J.-Y. Girard. Linear logic. Theoretical Computer Science 50, p. 1-102, (1987). 
\bibitem{H} R. Hartshorne. \emph{Algebraic Geometry}, Springer-Verlag, (1974).
\bibitem{J} P. T. Johnstone. \emph{Topos Theory}. Academic Press, (1977).
\bibitem{K} C. Kassel. {\em Quantum Groups}. Springer-Verlag, (1995).
\bibitem{Landi} G. Landi. {\it An introduction to noncommutative spaces and their geometries.} Lecture Notes in 
Physics,  Springer-Verlag, 
(1997).
\bibitem{CWM} S. MacLane, \emph{Categories for the Working Mathematician}, Springer-Verlag, (1971).
\bibitem{M} H. Matsumura. \emph{Commutative Ring Theory}, Cambridge University Press, (1986).
\bibitem{O} T. O'Neill.  {\it Differential Forms for T-Algebras in K\"ahler Categories.} M.Sc. Thesis, (2013). 
\bibitem{Pie} R.S. Pierce. {\it Associative Algebras}, Springer-Verlag, (1982).

\end{thebibliography}
\end{document}